\documentclass[a4paper, 14pt]{article}
\usepackage{amsmath}
\usepackage{amscd}
\usepackage{amsthm}
\usepackage{amssymb} \usepackage{latexsym}
\usepackage{eufrak}
\usepackage{euscript}
\usepackage{epsfig}
\usepackage{graphics}
\usepackage{array}
\usepackage{enumerate}
\usepackage{hyperref}

\usepackage{graphicx}
\usepackage{amsmath}
\usepackage{amsfonts}
\usepackage{amsthm}
\usepackage{mathrsfs}
\usepackage{dsfont}
\usepackage{indentfirst}
\usepackage{esint}
\usepackage{yhmath}
\numberwithin{equation}{section}

\newcommand{\R}{{\mathbb R}}

\newcommand{\N}{{\mathbb N}}

\newcommand{\lm}{{\lambda}}

\newcommand{\bel}[1]{\begin{equation}\label{#1}}

\newcommand{\bq}{\begin{equation}}
\newcommand{\ee}{\end{equation}}
\newcommand{\ba}{\begin{eqnarray}}
\newcommand{\ea}{\end{eqnarray}}

\newcommand{\qe}{\end{equation}}

\newtheorem{thm}{Theorem}[section]
\newtheorem{coro}[thm]{Corollary}
\newtheorem{lemma}[thm]{Lemma}
\newtheorem{pro}[thm]{Proposition}

\theoremstyle{definition}
\newtheorem{definition}[thm]{Definition}
\newtheorem{example}[thm]{Example}
\newtheorem{remark}[thm]{Remark}

\newcommand{\Hmm}[1]{\leavevmode{\marginpar{\tiny%
$\hbox to 0mm{\hspace*{-0.5mm}$\leftarrow$\hss}%
\vcenter{\vrule depth 0.1mm height 0.1mm width \the\marginparwidth}%
\hbox to 0mm{\hss$\rightarrow$\hspace*{-0.5mm}}$\\\relax\raggedright
#1}}}

\usepackage{color}

\date{}
\DeclareMathOperator{\erf}{erf}
\begin{document}

\title{Spectral classes of regular, random, and empirical graphs }

\author{Jiao Gu$^{1,2,}$\footnote{Corresponding author with E-mail address:
jiao@bioinf.uni-leipzig.de},
J{\"u}rgen Jost$^{1,3}$,
Shiping Liu$^{4}$, and
Peter F.\ Stadler$^{2,1,3,5}$\\[1em]
\begin{minipage}{1.0\textwidth}
\begin{small}
$^1$Max Planck Institute for Mathematics in the Sciences, Inselstra{\ss}e
    22, D-04103 Leipzig, Germany.\\
$^2$Bioinformatics Group, Department of Computer Science and
   Interdisciplinary Center for Bioinformatics,
   University of Leipzig, H{\"a}rtelstrasse 16-18, D-04103 Leipzig,
   Germany.\\
$^3$Santa Fe Institute, 1399 Hyde Park Rd., Santa Fe NM 87501, USA.\\
$^4$Department of Mathematical Sciences, Durham University,
   DH1 3LE Durham, United Kingdom.\\
$^5$Institute for Theoretical Chemistry, W{\"a}hringerstrasse 17, A-1090
   Wien, Austria.\\
\end{small}
\end{minipage}
}
\maketitle

\begin{abstract}
  We define a (pseudo-)distance between graphs based on the spectrum of the
  normalized Laplacian, which is easy to compute or to estimate
  numerically. It can therefore serve as a rough classification of large
  empirical graphs into families that share the same asymptotic behavior of
  the spectrum so that the distance of two graphs from the same family is
  bounded by $\mathcal{O}(1/n)$ in terms of size $n$ of their vertex sets.
  Numerical experiments demonstrate that the spectral distance provides a
  practically useful measure of graph dissimilarity.
  \vspace*{1cm}

\end{abstract}

\maketitle
\newpage

\section{Introduction}
Structural comparison of graphs has important applications in biology and
pattern recognition, see e.g. \cite{conte2004thirty,bunke1998graph}.  The
problem comes in two distinct flavors: it is comparably easy when
correspondences between nodes are known. This is the case e.g.\ for the
comparison of metabolic networks or protein-protein interaction networks
\cite{Berg:06}. The problem becomes much more difficult when node
correspondences are unknown, as is the case e.g.\ in the atom-mapping
problem reviewed in \cite{Chen:2013}. A classical combinatorial formulation
of the latter problem is to find the largest graph $G$ that is isomorphic
to a subgraph of each of two given input graphs $G_1$ and $G_2$. A natural
metric distance is given by $d_{MCSI}(G_1,G_2) :=\Vert G_1\setminus G\Vert
+ \Vert G_2\setminus G\Vert$, were $\Vert\cdot\Vert$ is measure of graph
size, e.g.\ the sum of edges and vertices. The main difficulty for
practical applications is that ``maximum common subgraph isomorphism
problem'' is NP-complete \cite{Cook:71,Read:77} and even APX-hard
\cite{Bahiense:12}.

For large graphs, thus, more computationally efficient distance measures
are required. Graph kernels \cite{Vishwanathan:10} describe graphs as
vectors of features, usually the occurrence data of small subgraphs have
increasingly been used in bioinformatics \cite{Kundu:13} and
chemoinformatics \cite{Ralaivola:05}. A related approach computed the
earth movement distance between the distributions of graph features
\cite{macindoe2010graph}. A practical difficulty is the fact that a very
large number of features is required to achieve sufficient resolution for
very large graphs.

Here we pursue a different approach that makes use of the representation of
graphs by its adjacency or its Laplacian matrix. Spectral properties of
these matrix representation are closely related to the graph structure
\cite{chung1997spectral,jost2001spectral,ranicki1992algebraic}. Spectral
graph theory in turn has received much inspiration from eigenvalue
estimates in Riemannian geometry, see e.g.\ \cite{chung1997spectral,
  bauer2012ollivier, horak2013interlacing}. Many of these estimates involve
only particular eigenvalues, like the smallest or largest. Here, instead we
wish to compare the entire spectra of two graphs to get some idea of how
similar or different they are \cite{banerjee2008spectrum}. The advantage of
such an approach is that nowadays, there exist very efficient and
numerically stable algorithms for computing the eigenvalues of a large
$(N\times N)$-matrix, in fact with an effort of only $\mathcal{O}(N^{2})$
in practice \cite{mario2012}. For the set of graphs of the same size, a
spectral distance based on the adjacency matrix was suggested in
\cite{stevanovic2007research} as a cospectral measure and further studied
in \cite{jovanovic2012spectral}. Although co-spectral graphs do exist, and
it remains an open problem what fraction of graphs is uniquely determined
by its spectrum \cite{Wang:10}, we shall see that the comparison of graph
spectra nevertheless provides as sensitive and computationally attractive
graph distance. We propose here a spectral distance associated with the
normalized Laplacian instead of the adjacency matrix, without any
constraint on the graph sizes. The reason is that the normalized Laplacian,
with its natural random walk or diffusion interpretation, seems to capture
some geometric properties better than the adjacency matrix.

Throughout this paper we assume that $G=(V,E)$ is a simple graph, i.e. a
finite, undirected, unweighted graph without self-loops or multiple edges,
where $V$ and $E$ are the vertex and edge sets. $n=|V|$ is the size of
$G$. We denote adjacency of $x$ and $y$ interchangeably as $x\sim y$ or
$xy\in E$.

The normalized Laplacian $\Delta_G$ of $G$ operates on functions $f: V\to
\R$ by
\begin{equation}
\Delta_G f(x)=f(x)-\frac{1}{d_x}\sum_{y\sim x}f(y),
\end{equation}
where $d_x$, the degree of $x$, is the number of edges connected to
$x$. $\Delta_G$ is a bounded and self-adjoint operator. The definition of
$\Delta$ can be expanded without difficulty to weighted graphs and with
some more effort also to directed ones. We will not explicitly consider
these more general cases here.

The spectrum of $\Delta_{G}$, denoted by $\lambda(G)$, consists of $n$
eigenvalues, all contained in $[0,2]$, i.e.,
$0=\lambda_{0}\leq\lambda_{1}\leq\ldots\leq\lambda_{n-1}\leq 2$. This
yields the Radon probability measure
\begin{equation}
 \label{eq:01}
\mu(G)=\frac{1}{n}\sum_{i=0}^{n-1} \delta_{\lm_i(G)}
\end{equation}
on $[0,2]$, where $\delta_{\cdot}$ denotes the Dirac measure.  We can then
integrate functions against this measure, for instance a Gaussian kernel
with center $x$ and bandwidth $\sigma$, to obtain
\cite{rosenblatt1956remarks, parzen1962estimation},
 \begin{equation}
  \label{eq:02}
  \rho_{G}(x)=\frac{1}{n}\sum\limits_{i=0}^{n-1}
  {\frac{1}{\sqrt{2\pi\sigma^{2}}}
    e^{-\frac{(x-\lambda_{i})^{2}}{2\sigma^{2}}}}.
 \end{equation}
 $\rho_{G}$ thus is a smoothed spectral density of $G$. It naturally gives
 rise to a pseudometric on the space of simple graphs by means of the
 $\ell^{1}$ distance of the spectral densities:
\begin{equation}
D(G,G'):= \int|\rho_{G}(x)-\rho_{G'}(x)|dx.
\label{e:distance}
\end{equation}

\begin{remark} It is obvious that $D$ yields a pseudometric on the space of
  (isomorphism classes of) graphs. $D$ is not a distance because of the
  possibility of cospectral graphs, that is, non-isomorphic graphs with the
  same spectrum, see \cite{mario2012} for a survey. Nevertheless, we shall
  call $D(G,G')$ the ``spectral distance'' between the graphs $G$ and
  $G'$. \end{remark}

\begin{remark}
  The fact that all complete bipartite graphs $K_{n_1,n_2}$ with the same
  total number $n=n_1+n_2$ of vertices have the same spectrum (it consists
  of 0 and 2 with multiplicity 1 and of 1 with multiplicity $n-2$) shows
  that the spectrum of the normalized Laplacian is not sensitive to the
  number of edges. This can be easily remedied, however, by using the
  (normalized) Laplacian on edges rather than on vertices. Then, see
  \cite{horak2013simplicial}, the spectrum is the same as that for
  vertices, except for the multiplicity of the eigenvalue zero which now
  counts the number of independent cycles, that is, is equal to
  $|E|-|V|+1$. In contrast, for the Laplacian on vertices that we have used
  here, the multiplicity of the eigenvalue 0 is equal to the number of
  connected components. All subsequent constructions will work for the
  Laplacian on edges as for that on vertices.
\end{remark}

In Section 2, the asymptotic behavior of families $G_n$ of finite graphs is
discussed, when their number $n$ of vertices tends to infinity. It turns
out that typical classes of graphs, like complete, complete bipartite,
cycle, path, cube, have characteristic asymptotic properties of the
corresponding measures $\mu$ from \eqref{eq:01}. In section 3, using the
interlacing theory, we prove that the distance between two finite graphs of
size of order $n$ that differ only in finitely many edit operations is
equal to $\mathcal{O}(1/n)$. Moreover, we show that the distance for
particular pairs of graphs converges to zero even with speed with
$\mathcal{O}(1/n^2)$ in many cases. Finally, in the numerical part of this
paper, we compare spectral distances within families of random graphs.

\section{Spectral classes of graphs}
Empirical studies have shown that qualitatively different type of large
graphs can in many cases be distinguished by the shape of their spectral
density. For example, in 1955, Wigner introduced his famous semicircle
law, which says that the spectrum of a large random symmetric matrix
follows a semicircle distribution \cite{wigner1955characteristic,
  wigner1956characteristic, wigner1958distribution}. In \cite{banerjee2008spectral}, it was found that the spectral distribution is an important characteristic of a
network, and  a  classification scheme for empirical networks  based on the
spectral plot of the Laplacian of the graph underlying the network was introduced.

Let $(G_n)_{n\in \mathbb{N}}$ be an infinite family of graphs $G_n$
of $n$ vertices.  An important recent development in graph theory is
concerned with the construction of suitable limits of such
families for $n\to \infty$. Typically, such limits should reflect the asymptotic distribution
of isomorphism classes of subgraphs. For dense graphs, one obtains the
graphons, whereas for sparse graphs, one has the notion of graphings, see
Lov{\'a}sz' monograph \cite{lovasz2012} for an overview.  Here, we propose
a weaker notion that is based on graph spectra, more precisely on the Radon
measure defined in (\ref{eq:01}). Thus, for a continuous function
$f:[0,2]\to \R$, we have
\begin{equation}
\mu(G)(f)=\frac{1}{n}\sum_{i=0}^{n-1} f(\lm_i(G)).
\label{eq:06}
\end{equation}
Recall that a family $\mu_n$ of Radon measures on $[0,2]$ converges weakly
to the Radon measure $\mu_0$, in symbols $\mu_n \rightharpoondown \mu_0$,
if $\mu_0(f)=\lim_{n\to \infty}\mu_n(f)$ for all continuous functions
$f:[0,2]\to \R$.

\begin{definition}
  A family $(G_n)_{n\in \N}$ of graphs belongs to the spectral class
  $\rho$, where $\rho$ is a Radon measure on $[0,2]$, if
  $\mu(G_n) \rightharpoondown \rho$ for $n\to \infty$.
\end{definition}
To demonstrate that this definition is meaningful, we consider a few simple
examples.

\begin{pro}
  Let $G_n=K_n$, the complete graph on $n$ vertices, or $G_n=K_{n_1,n_2}$,
  the complete bipartite graph on $n_1+n_2=n$ vertices. Then $G_n$ belongs
  to the spectral class $\delta_1$, where $\delta_1$ is the Dirac measure
  supported at $1$.
\label{complete_complete_bipartite}
\end{pro}
In particular, the complete and the complete bipartite graphs
asymptotically belong to the same spectral class.
\begin{proof}
  The spectrum of $K_n$ consists of $0$ with multiplicity 1 and of
  $\frac{n-1}{n}$ with multiplicity $n-1$, and that of $K_{n_1,n_2}$ of $0$
  and $2$ with multiplicity 1 each and $1$ with multiplicity $n-2$. This
  easily implies the result.
\end{proof}
In fact, the family of $n$-cubes also asymptotically belongs
 to the same spectral class.

\begin{pro}
  Let $G_n$ be the $n$-cube on $2^n$ vertices, then the corresponding
  spectral class is given by $\delta_1$.
\label{cube_cube}
\end{pro}
\begin{proof}
  The spectrum of the $n$-cube consists of $\frac{2k}{n}$ with multiplicity
  $\binom{n}{k}$, where $k=0,\ldots,n$. For any continuous function $f:
  [0,2]\rightarrow \mathbb{R}$, $\forall \varepsilon>0$, there exists a
  sufficiently small constant $\alpha\in(0,\frac{1}{2})$, such that
  $|f(x)-f(1)|<\varepsilon$ for any $x\in (1-2\alpha,1+2\alpha)$. Then,
\begin{align*}
  &\left|\lim_{n\to \infty}\sum\limits_{k=0}^{n}\frac{1}{2^n}
    \binom{n}{k}f\left(\frac{2k}{n}\right)-f(1)\right|\\
  \leq&\left|\lim_{n\to \infty}
    \sum\limits_{k=0}^{[\frac{n}{2}(1-\alpha)]-1}
    \frac{1}{2^n}\binom{n}{k}f\left(\frac{2k}{n}\right) \right|+
  \left| \lim_{n\to \infty}\sum\limits_{k=[\frac{n}{2}(1+\alpha)]+1}^{n}
    \frac{1}{2^n}\binom{n}{k}f\left(\frac{2k}{n}\right) \right|\\
  &+\left|\lim_{n\to \infty}
    \sum\limits_{k=[\frac{n}{2}(1-\alpha)]}^{[\frac{n}{2}(1+\alpha)]}
    \frac{1}{2^n}\binom{n}{k}f\left(\frac{2k}{n}\right) -f(1)\right|\\
  \leq&\max\limits_{x\in[0,2]}\left|f(x)\right|\lim_{n\to \infty}
  \sum\limits_{k=0}^{[\frac{n}{2}(1-\alpha)]-1}
  \frac{1}{2^n}\binom{n}{k}+\max\limits_{x\in[0,2]}\left|f(x)\right|
  \lim_{n\to \infty}\sum\limits_{k=[\frac{n}{2}(1+\alpha)]+1}^{n}
  \frac{1}{2^n}\binom{n}{k} \\
  &+\left|\lim_{n\to \infty}
    \sum\limits_{k=[\frac{n}{2}(1-\alpha)]}^{[\frac{n}{2}(1+\alpha)]}
    \frac{1}{2^n}\binom{n}{k}f\left(\frac{2k}{n}\right)-f(1)\right|.
\end{align*}
With Stirling's approximation $n!\thickapprox\sqrt{2\pi n}(\frac{n}{e})^n$ \cite{robbins1955remark}
and $t:=\frac{1-\alpha}{2}$, we obtain
\begin{align*}
&\lim\limits_{n\to \infty}\sum\limits_{k=0}^{[\frac{n}{2}(1-\alpha)]-1}
\frac{1}{2^n}\binom{n}{k}\leq\lim_{n\to \infty}
\frac{nt}{2^n}\binom{n}{nt-1}
\leq\lim_{n\to \infty}\frac{nt}{e^{n\log2}}\binom{n}{nt}\\
=&\lim_{n\to \infty}\frac{t\sqrt{n}}{\sqrt{2\pi}e^{n\log 2}}
\frac{1}{e^{nt\log t+n(1-t)\log(1-t)+\frac{1}{2}\log t+\frac{1}{2}\log(1-t)}}
\end{align*}
Suppose $g(s)=s\log s+(1-s)\log(1-s)$ for $s\in(0,1)$, then $g'(s)=\log
s-\log(1-s)$, we know $g(s)$ attains its minimal value $-\log 2$ at
$\frac{1}{2}$. Then we have $g(\frac{1-\alpha}{2})=-\log 2+C_{\alpha}$ for
some positive constant $C_{\alpha}$ depending on $\alpha$. Therefore,
$\lim\limits_{n\to \infty}
\sum\limits_{k=0}^{[\frac{n}{2}(1-\alpha)]-1}\frac{1}{2^n}\binom{n}{k}=0$.
Similarity, we also can prove
$\lim\limits_{n\to \infty}
\sum\limits_{k=[\frac{n}{2}(1+\alpha)]+1}^{n}\frac{1}{2^n}\binom{n}{k}=0$.
Combined with $\sum\limits_{k=0}^{n}\frac{1}{2^n}\binom{n}{k}=1$, we have,
\begin{align*}
&\left|\lim_{n\to \infty}\sum\limits_{k=0}^{n}
  \frac{1}{2^n}\binom{n}{k}f(\frac{2k}{n})-f(1)\right|
\leq\left|\lim_{n\to \infty}
  \sum\limits_{k=[\frac{n}{2}(1-\alpha)]}^{[\frac{n}{2}(1+\alpha)]}
  \frac{1}{2^n}\binom{n}{k}f(\frac{2k}{n}) -f(1)\right|\\
\leq&\left|\lim_{n\to \infty}
  \sum\limits_{k=[\frac{n}{2}(1-\alpha)]}^{[\frac{n}{2}(1+\alpha)]}
  \frac{1}{2^n}\binom{n}{k}\left(f(\frac{2k}{n}) -f(1)\right)\right|\\
\leq&\lim_{n\to \infty}
  \sum\limits_{k=[\frac{n}{2}(1-\alpha)]}^{[\frac{n}{2}(1+\alpha)]}
  \frac{1}{2^n}\binom{n}{k}\varepsilon\leq\varepsilon
\end{align*}
So, we have
$$\lim_{n\to \infty}\sum\limits_{k=0}^{n}\frac{1}{2^n}
\binom{n}{k}f(\frac{2k}{n})=f(1).$$
Hence, for $G_{n}$, the corresponding spectral class is $\delta_1$.
\end{proof}

\begin{pro}
  Let $G_n$ be the petal graph on $n=2m+1$ vertices, that is, the graph
  consisting of $m$ triangles all joined at one vertex (this vertex then
  has degree $2m$, whereas all other vertices are of degree 2). The
  spectral class then is given by
  $\frac{1}{2}\delta_{\frac{1}{2}}+\frac{1}{2}\delta_{\frac{3}{2}}$.
\end{pro}
\begin{proof}
  The spectrum of $G_n$ consists of $0$ with multiplicity 1, $\frac{1}{2}$
  with multiplicity $m-1$ and $\frac{3}{2}$ with multiplicity $m+1$.
\end{proof}

\begin{pro}
  In contrast, when $G_n=P_n$, the path with $n$ nodes, or $G_n=C_n$, the
  cycle with $n$ nodes, then the corresponding spectral class $\rho$ has no
  atoms, that is, $\rho(A)=0$ whenever $A$ is a finite subset of $[0,2]$,
  for instance a single point. In fact, $\rho$ is absolutely continuous
  with respect to the Lebesgue measure on $[0,2]$ with a probability
  density function (see Fig. \ref{F1:2.1 density})
  \begin{equation}
  g(x)=\frac{1}{\pi}\frac{1}{\sqrt{2x-x^2}}.
  \end{equation}
  \label{path_cycle}
\end{pro}

\begin{figure}[t]
\begin{center}
\includegraphics[width=3in]{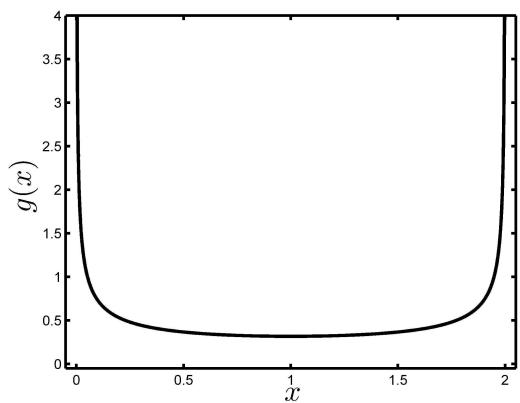}
\end{center}
\caption{The probability density function $g(x)$.}
\label{F1:2.1 density}
\end{figure}
\begin{proof}
  The spectrum of $P_n$ consists of $1-\cos\frac{\pi k}{n-1}$ with equal
  multiplicity 1, where $k=0, \ldots, n-1$. Denote the cumulative
  distribution functions of $\mu(P_n)$ by $F_n(x):=\mu(P_n)([0,x])$ for
  $0\leq x\leq 2$. We observe that
\begin{equation}
  F_n(x)=\frac{1}{n}
  \left(\left[\frac{n-1}{\pi}\arccos(1-x)\right]+1\right)
\end{equation}
and then
\begin{equation}
\lim_{n\rightarrow\infty}F_n(x)=\frac{1}{\pi}\arccos(1-x):=F(x).
\end{equation}
Let $\rho$ be the probability measure such that $d\rho(x)=g(x)dx$, $x\in
[0,2]$. Then $F$ is the cumulative distribution function of $\rho$.  By the
property of weak convergence of probability measures on $\mathbb{R}$, we
know $\mu(P_n)\rightharpoondown\rho$. The case for $C_n$ with the spectrum
$1-\cos\frac{2\pi k}{n}$, where $k=0,\ldots, n-1$, can be proved
similarly. They belong to the same spectral class $\rho$.
\end{proof}

There also exist families $(G_n)$ whose asymptotic spectral contains both a
Dirac and a regular component. For instance, take the graph $G_n$ with
$n=3m$ obtained from an even cycle $C_{2m}$ where every other node is
duplicated. According to \cite{banerjee-jost2008}, this graph has the
eigenvalue $1$ with multiplicity $m$, but it also inherits a slightly
perturbed version of the spectrum of $C_{2m}$. The latter yields a regular
contribution to the asymptotic spectral measure, whereas the former
contributes an atomic part $\frac{1}{3}\delta_1$.

Given an arbitrary family of graphs, the corresponding spectral class may
be not well-defined.
\begin{example}
Let $G_n$ be given as follows.
\begin{equation*}
G_n=\left\{
      \begin{array}{ll}
        K_n, & \hbox{if $n$ even;} \\
        \text{petal graph}, & \hbox{if $n$ odd.}
      \end{array}
    \right.
\end{equation*}
Then there is no well-defined spectral class for $(G_n)_{n\in \mathbb{N}}$.
\end{example}
However, by the Prokhorov theorem \cite{Prokhorov:56}, we know that for any
family $(G_n)$ of graphs, there at least exists a sub-family of it which
belong to one spectral class.

The main result of this section is that two graph families that differ only
by finite modifications do not differ in their asymptotic spectral
class. Normally, two graphs could be related to each other through some
modification. We define the following modification as the edit operations.

\begin{definition}
  An edit operation on a graph $G$ is the insertion or deletion of an edge
  or the insertion or deletion of an isolated vertex.
\end{definition}
A finite graph always can be changed to  another finite one through
finitely many edit operations. If $G'$ and $G''$ are connected there is
always a sequence of edit operations so that all intermediates are also
finite. Furthermore, one may ``couple'' the insertion and deletion of
isolated vertices with the insertion of the first and the deletion of the
last edges that it is incident with, so that graph editing can be
specified in terms of edge insertions and edge deletions. The edit distance
$d_{\textrm{edit}}(G',G'')$ between two graphs $G'$ and $G''$ can be defined
as the minimal number of edge operations required to convert $G'$ into
$G''$. It is well known that $d_{\textrm{edit}}$ is a metric and equals
$|E(G')|+|E(G'')|-2 |E(H)|$, where $H$ is the edge-maximal common subgraph
of $G'$ and $G''$, which is very hard to compute.

\begin{thm}\label{seciton2main}
  Let $G_n$ and $G'_{n'}$ be graphs with $n$ and $n'$ vertices,
  respectively. Assume that $G'_{n'}$ can be obtained from $G_n$ by at most
  $C$ steps of edit operations, where $C$ is independent of $n$. Then the
  families $(G_n)$ and $(G'_{n'})$ belong to the same spectral class
  (assuming that the corresponding spectral measures possess weak limits).
\label{thm:interlacing1}
\end{thm}

Theorem \ref{thm:interlacing1} is a consequence of the interlacing
properties of the Lapacian spectra
\cite{haemers1995interlacing,horak2013interlacing,chen2004interlacing} for
very similar graphs. More precisely, the eigenvalues of two graphs $G$ and
$G'$ control each other by virtue of inequalities of the form
\begin{equation}
\lambda_{i-t_{1}}\leq \lambda'_{i}\leq\lambda_{i+t_{2}},
\end{equation}
where the integers $t_{1}$ and $t_{2}$ are independent of the index $i$ but
explicitly depend on the topological characteristics of the operation
required to convert $G$ to $G'$. Interlacing properties for the normalized
Laplacian spectra were studied in particular in
\cite{chen2004interlacing}.

\begin{lemma} Let $G$ and $G'$ be two graphs as in Theorem
\ref{thm:interlacing1} and let
$\lambda_{0}\leq\lambda_{1}\leq\cdots\leq\lambda_{n-1}$ and
$\theta_{0}\leq\theta_{1}\leq\cdots\leq\theta_{n'-1}$
be the eigenvalues of $\Delta_G$ and $\Delta_{G'}$, respectively.
We then have a constant $C$ so that
\begin{equation}\label{interlacing}
 \lambda_{i-C}\leq\theta_{i}\leq\lambda_{i+C}
\end{equation}
holds for $i=0,1,\ldots,n-1$, where we used the notations $\lambda_j=0$ for
$j<0$ and $\lambda_j=2$ for $j\ge n$.
\label{lem:0}
\end{lemma}
\begin{proof}
  Without loss of generality, we suppose $n'\geq n$. Then we can obtain
  $G'$ from $G$ by firstly adding $(n'-n)$ isolated vertices, then deleting
  $t_1$ edges and finally adding $t_2$ edges. Note adding $(n'-n)$ isolated
  vertices to $G'$ produces $(n'-n)$ additional zero eigenvalues. Set
  $C:=n'-n+t_1+t_2$. Then (\ref{interlacing}) is a direct corollary of
  Theorem 2.3 in \cite{chen2004interlacing}, where the interlacing
  inequalities for deleting one edge is proved.
\end{proof}

With the above lemma, we can prove Theorem \ref{thm:interlacing1}.

\begin{proof}[Proof of Theorem \ref{seciton2main}]
  Suppose $G$ and $G'$ have normalized Laplacian spectra
$$0=\lambda_{0}\leq\lambda_{1}\leq\ldots\leq\lambda_{n-1}\leq 2
\mbox{ and } 0=\theta_{0}\leq\theta_{1}\leq\ldots\leq\theta_{n'-1}\leq 2,$$
respectively. Without loss of generality, we suppose $n'-n=C'\leq
C$. According to Lemma \ref{lem:0}, we
have $$\lambda_{i-C}\leq\theta_{i}\leq\lambda_{i+C}.$$

Let $f: [0,2]\rightarrow \mathbb{R}$ be a continuous function. By
approximation, we may assume that $f$ is differentiable. $f$ can be
decomposed into the sum $f=\tilde{f_{+}}+\tilde{f_{-}}$ of a monotonically
increasing function $\tilde{f_{+}}=\int_{0}^{x}(f')^{+}+f(0)$ and a
monotonically decreasing function $\tilde{f_{-}}=\int_{0}^{x}(f')^{-}$. For
continuous monotonic functions $F$ in $[0,2]$, we have,
$$\lim_{n\to \infty}{ \left(\frac{1}{n}\sum\limits_{i=0}^{n-1}F(\lambda_{i})-\frac{1}{n}\sum\limits_{i=-C}^{n-C-1}F(\theta_{i})\right)
 }\geq 0$$
and
\begin{equation*}
\lim_{n\to \infty}{ \left(\frac{1}{n}\sum\limits_{i=0}^{n-1}F(\lambda_{i})-\frac{1}{n}\sum\limits_{i=C}^{n+C-1}F(\theta_{i})\right) }\leq 0.
\end{equation*}
Because $C'\leq C$ is bounded and independent on $n$, and $F(x)$ is always
bounded, we have,
$$\lim_{n\to \infty}{ \left(\frac{1}{n}\sum\limits_{i=0}^{n-1}F(\lambda_{i})-\frac{1}{n'}\sum\limits_{i=0}^{n'-1}F(\theta_{i})\right) }=0.$$
This then also holds for $f$ because of the above decomposition. Therefore,
the families $(G_{n})$ and $(G'_{n})$ belong to the same spectral class.
\end{proof}

\section{The spectral distance on general graphs}

In this section, we explore the properties of the spectral distance between
two related finite graphs $G$ and $G'$, i.e., $G'_{n'}$ can be obtained
from $G_n$ by $C$ steps of edit operations as in Theorem
\ref{thm:interlacing1}. If the number of the edit operations is bounded by
a constant which is independent of the graph size, the spectral distance between
a graph and its editing graph tends to zero when their sizes tend to
infinity. We start with

\begin{thm}
  Let $(G_n)_{n\in \N}$ and $(G'_{n'})_{n'\in \N}$ be two families of
  graphs that belong to the same spectral class $\mu_0$. Then
  $\lim_{n\rightarrow \infty}D(G_n, G'_{n'})=0$.
\label{thm:spectral}
\end{thm}
\begin{proof}
  Let $\mu(G_n)$, $\mu'(G'_{n'})$ be the spectral measures of $G_n$,
  $G'_{n'}$, and
$$f(x,\lambda):=\frac{1}{\sqrt{2\pi\sigma^2}}e^{-\frac{(x-\lambda)^2}{2\sigma^2}}.$$
For fixed $x$, we also write $f^x(\lambda):=f(x,\lambda)$ to indicate that it is a function of the variable $\lambda$. Then recalling (\ref{eq:06}), we have
$$\mu(G_n)(f^x)=\rho_G(x),$$
where $\rho_G(x)$ is the kernel function in (\ref{eq:02}).  Therefore we
have by the definition of graph distance
\begin{equation}\label{addedproof}
D(G_n, G'_{n'})=\int|\rho_{G_n}(x)-\rho_{G'_{n'}}|dx=
   \int|\mu(G_n)(f^x)-\mu(G'_{n'})(f^x)|dx.
\end{equation}
Applying Lebesgue's dominated convergence theorem yields
\begin{align*}
\lim_{n\rightarrow \infty}D(G_n, G'_{n'})&=\lim_{n\rightarrow \infty}
\int|\mu(G_n)(f^x)-\mu(G'_{n'})(f^x)|dx\\
&=\int|\lim_{n\rightarrow\infty}(\mu(G_n)(f^x)-\mu(G'_{n'})(f^x))|dx\\
&=\int|\mu_0(f^x)-\mu_0(f^x)|dx=0.
\end{align*}
\end{proof}

Theorem~\ref{thm:spectral} states that if the corresponding spectral
measures of two related families $(G_{n})$ and $(G'_{n'})$ have the same
weak limit, then their spectral distance tends to zero. Even when this
condition might not be satisfied, this conclusion can still hold for
certain functions.

\begin{thm}
  $G$ and $G'$ are graphs of size $n$ and $n'$ respectively. Assume that
  $G'$ can be obtained from $G$ by at most $C$ steps of edit operations,
  where $C$ is independent of $n$. Then $D(G,G')=\mathcal{O}(1/n)$ as
  $n\rightarrow \infty$.
\label{thm:interlacing}
\end{thm}

\begin{proof} Suppose $G$ and $G'$ have normalized Laplacian spectra
$$0=\lambda_{0}\leq\lambda_{1}\leq\ldots\leq\lambda_{n-1}\leq 2
\hspace*{1cm} and \hspace*{1cm}
0=\theta_{0}\leq\theta_{1}\leq\ldots\leq\theta_{n'-1}\leq 2,$$
respectively. Without loss of generality, $n'\geq n$.

Then, the spectral distance $D(G,G')$ (recall \ref{e:distance}) is:
\begin{align*}
D(G,G')=&\int{\left|\rho_{G}(x)-\rho_{G'}(x)\right|dx}\\
\leq &\frac{1}{\sqrt{2\pi}\sigma}\int\left|
  \frac{1}{n}\sum\limits_{i=0}^{n-1}
  e^{-\frac{(x-\lambda_{i})^{2}}{2\sigma^{2}}}-
  \frac{1}{n}\sum\limits_{i=0}^{n-1}
  e^{-\frac{(x-\theta_{i})^{2}}{2\sigma^{2}}}\right|dx\\
&+\frac{1}{\sqrt{2\pi}\sigma}\int\left|\frac{1}{n}\sum\limits_{i=0}^{n-1}
  e^{-\frac{(x-\theta_{i})^{2}}{2\sigma^{2}}}-
  \frac{1}{n'}\sum\limits_{i=0}^{n-1}
  e^{-\frac{(x-\theta_{i})^{2}}{2\sigma^{2}}}-
  \frac{1}{n'}\sum\limits_{i=n}^{n'-1}
  e^{-\frac{(x-\theta_{i})^{2}}{2\sigma^{2}}}\right|dx\\
&:=\text{I}+\text{II}.
\end{align*}
Clearly, $\text{II}\leq(1/n-1/n')n+(n'-n)/n'\leq 2(n'-n)/n'$.

By the mean value theorem, we see that there exists $\lambda_i^0$ between
$\lambda_i$ and $\theta_i$ such that
\begin{align*}
 \text{I}&\leq\frac{1}{n\sigma\sqrt{2\pi}}\sum_{i=0}^{n-1}
 \int\frac{|x-\lambda_i^0|}{\sigma^2}
 e^{-\frac{(x-\lambda_i^0)^2}{2\sigma^2}}|\lambda_i-\theta_i|dx\\
 &=\frac{\sqrt{2}}{n\sigma\sqrt{\pi}}\sum_{i=0}^{n-1}|\lambda_i-\theta_i|.
\end{align*}
Recall by Lemma \ref{lem:0}, we have
$\lambda_{i-C}\leq\theta_{i}\leq\lambda_{i+C}.$ where $C\geq n'-n$.  Then,
we obtain
\begin{align*}
  D(G,G')\leq \text{I}+\text{II}\leq
  \frac{\sqrt{2}}{n\sigma\sqrt{\pi}}
  \sum_{i=0}^{n-1}(\lambda_{i+C}-\lambda_{i-C})+\frac{2C}{n}.
\end{align*}
Further observing that, when $n$ is large,
\begin{equation*}
  \sum_{i=0}^{n-1}(\lambda_{i+C}-\lambda_{i-C})=
  \sum_{j=n-C-1}^{n-1}\lambda_j-\sum_{j=0}^C\lambda_j\leq 2C,
\end{equation*}
completes the proof.
\end{proof}

\begin{coro}
  $G$ and $G'$ are graphs of size $n$ and $n+C$, where $C$ is independent
  of $n$. If the vertex degrees of these two graphs are bounded by a
  constant independent of $n$, then $D(G,G')=\mathcal{O}(1/n)$ as
  $n\rightarrow \infty$.
\label{coro:interlacing}
\end{coro}
\begin{proof}
  Since the vertex degree in bounded by a constant, the number of edit
  operations required to obtain $G'$ from $G$ is bounded by a constant that
  is independent of the graph sizes. This corollary then follows directly
  from Theorem \ref{thm:interlacing}.
\end{proof}

\section{The spectral distance for particular graph classes}
\label{particular}

In this section, we give some examples for estimating the spectral
distances between graphs in particular classes.
\begin{example} For two star graphs $G$ and $G'$ with $n$ and $n'$ vertices,
  the spectral distance $D(G,G')$ is proportional to the difference of
  their average degree, i.e.,
  $D(G,G')\propto|\frac{2(n-1)}{n}-\frac{2(n'-1)}{n'}|$.
  \label{T:complete_bipartite}
\end{example}
\begin{proof}
  The difference of the average degrees of star graphs $G$ and $G'$
  is $$T(G,G')=\left|\frac{2(n-1)}{n}-\frac{2(n'-1)}{n'}\right|=
  2\left|\frac{1}{n'}-\frac{1}{n}\right|.$$

  Recall that the normalized Laplacian spectra for two complete bipartite
  graphs, hence in particular for star graphs $G$ and $G'$ are
  $\{0,1^{\{n-2\}},2\}$ and $\{0,1^{\{n'-2\}},2\}$ respectively. Here the
  index of an eigenvalue indicates its multiplicity.

  Then the spectral distance $D(G,G')$ (recall (\ref{e:distance})) is
  \begin{align*}
    D(G,G')=&\int{\left|\rho_{G}(x)-\rho_{G'}(x)\right|dx}\\
    =&\frac{1}{\sqrt{2\pi}\sigma}\int
    \bigg|\frac{1}{n}\left(e^{-\frac{x^{2}}{2\sigma^{2}}}+
      (n-2)e^{-\frac{(x-1)^{2}}{2\sigma^{2}}}+
      e^{-\frac{(x-2)^{2}}{2\sigma^{2}}}\right)\\
    &-\frac{1}{n'}\left(e^{-\frac{x^{2}}{2\sigma^{2}}}+
      (n'-2)e^{-\frac{(x-1)^{2}}{2\sigma^{2}}}+
      e^{-\frac{(x-2)^{2}}{2\sigma^{2}}}\right)\bigg|dx\\
    =&\frac{1}{\sqrt{2\pi}\sigma}
    \left|\frac{1}{n}-\frac{1}{n'}\right|
    \int{\left|e^{-\frac{x^{2}}{2\sigma^{2}}}+
        e^{-\frac{(x-2)^{2}}{2\sigma^{2}}}-
        2e^{-\frac{(x-1)^{2}}{2\sigma^{2}}}\right|}dx\\
    \propto &2\left|\frac{1}{n}-\frac{1}{n'}\right|=T(G,G')
\end{align*}
\end{proof}

\begin{example} For two complete bipartite graphs or two complete graphs
  $G$ and $G'$ with $n$ and $n+C$ vertices, where $C$ is independent of
  $n$, we have, $D(G,G')=\mathcal{O}(1/n^{2})$, as $n\rightarrow \infty$.
\label{C:complete_bipartite2}
\end{example}
\begin{proof} For two complete bipartite graphs, this follows directly from
  Example \ref{T:complete_bipartite}. Indeed, we have
\begin{align*}
D(G,G')
\propto&\left|\frac{1}{n}-\frac{1}{n+C}\right|
=\mathcal{O}\left(\frac{1}{n^{2}}\right).
\end{align*}
Recall the spectrum of a complete graph $G$ with $n$ nodes is
$\{0,n/(n-1)^{\{n-1\}}\}$.  Then, their spectral distance $D(G,G')$ (recall
(\ref{e:distance})) is:
\begin{align*}
  D(G,G')=&\int{|\rho_{G}(x)-\rho_{G'}(x)|dx}\\
  =&\frac{1}{\sqrt{2\pi}\sigma}\int\bigg|\left(\frac{1}{n}-\frac{1}{n'}\right)
  e^{-\frac{x^{2}}{2\sigma^{2}}}+\left(1-\frac{1}{n}\right)
  \left(e^{-\frac{(x-n/(n-1))^{2}}{2\sigma^{2}}}-
    e^{-\frac{(x-n'/(n'-1))^{2}}{2\sigma^{2}}}\right)\\
  &+\left(1-\frac{1}{n'}-1+\frac{1}{n}\right)
  e^{\frac{(x-n'/(n'-1))^{2}}{2\sigma^{2}}}\bigg|dx\\
  \leq &2\left(\frac{1}{n}-\frac{1}{n'}\right)+ \frac{1}{\sqrt{2\pi}\sigma}
  \int\left|e^{-\frac{(x-n/(n-1))^2}{2\sigma^2}}-
    e^{-\frac{(x-n'/(n'-1))^{2}}{2\sigma^{2}}}\right|dx.\\
\end{align*}
Now applying the mean value theorem as in the proof of Theorem
\ref{thm:interlacing} and using $n'=n+C$, we arrive at
$D(G,G')=\mathcal{O}\left(\frac{1}{n^2}\right)$ as $n\rightarrow\infty$.
\end{proof}

The above two examples show the behavior of the spectral distance between
two graphs differing in finite vertices. Next, we discuss some cases when
the difference in their sizes tends to infinity.
\begin{example} Let $G$ and $G'$ be graphs with $n$ and $n'$ vertices so that
$n\le n' \le kn$ for some constant $k\in \N$. Then the spectral distance
$D(G,G')$ tends to zero when $n\rightarrow \infty$ in the following cases:
\begin{enumerate}
\item $G$ and $G'$ are complete or complete bipartite.
\item $G$ and $G'$ are cycles or paths.
\item $G$ is an $m-$cube and $G'$ an $(m+\ell)-$cube,
  for $m,\ell \in \N$ ($m\to \infty$, $\ell$ fixed).
\end{enumerate}
\label{T1:cube}
\end{example}
\begin{proof}
  By Theorem \ref{thm:spectral}, the spectral distance tends to zero if the
  corresponding spectral classes are the same. The statements then follow
  from Propositions \ref{complete_complete_bipartite}, \ref{path_cycle} and
  \ref{cube_cube}.
\end{proof}

The Examples \ref{C:complete_bipartite2} show that if two graphs belong to a
certain restricted class (complete or complete bipartite), then their
spectral distance decreases as $\mathcal{O}(n^{-2})$ for $n\rightarrow
\infty$, i.e., it converges to zero more quickly than the
$\mathcal{O}(n^{-1})$ bound of Theorem \ref{thm:interlacing}. A similar
result holds for cubes. From Example \ref{T1:cube}, the spectral distance
between two cubes tends to 0, when their sizes tend to infinity. Actually,
the spectral distance between the $(n-1)-$cube and the $n-$cube is less than
$2\erf\left(\frac{2}{n-1}\cdot\frac{1}{2\sigma\sqrt{2}}\right)$, where
$\erf(x):=\frac{2}{\sqrt{\pi}}\int_0^xe^{-t^2}dt$ is the Gauss error
function. Their size difference is $2^{n-1}$, which grows to infinity when
$n$ tends to infinity. However, if we scale their sizes as $n$ and $n+C$
respectively, then their spectral distance is equal to $\mathcal{O}(1/n\log
n)$ when $n\rightarrow \infty$, which converges to zero even more quickly
than $\mathcal{O}(1/n)$.

\section{Computational results for generic graphs}

In the previous section we have seen that the spectral distance between two
graphs from the same class tends to be small. We complement these results
here by a numerical investigation of some generic classes of graphs such
as  $k-$regular trees, random or scale-free graphs. The main motivation
for studying the spectral distance is its potential use as a means of
discriminating graphs in dependence on their structural differences. This
begs the question, of course, what we mean by structural difference in the
first place. For classes of random graphs, in particular, it seems natural
to require that this structure should be more or less independent of the
size, making the Radon measures discussed above an attractive choice.
Alternative classification schemes have of course been used in the
literature. Common schemes of generation, as in the Barab{\'a}si-Albert
scale-free model \cite{barabasi1999emergence,dorogovtsev2000structure} are
one possibility. Graphs generated by different models, e.g.\
Erd{\H{o}}s-R{\'e}nyi random graphs, should then be assigned to a different
class. Does our spectral distance measure reflect such differences?

As a further motivation, recall Wigner's famous semicircle law for the
asymptotic spectrum of random matrices. More specifically, the spectrum of
the normalized Laplacian of a random graph (except for the
eigenvalue 0, which asymptotically has a finite measure) converges to a
semicircle with radius $r=2/\sqrt{\overline{w}}$ \cite{chung2003spectra}.
A sufficient condition for this result is the assumption that the expected
minimum degree $d_{\min}$ is much larger than $\sqrt{\overline{w}}$, where
$\overline{w}$ is the expected average degree. This condition is satisfied
in particular by Erd{\H{o}}s-R\'{e}nyi random graphs and scale-free
graphs. Most graphs generated by the Erd{\H{o}}s-R\'{e}nyi random model satisfy
the condition $d_{\min}\gg \sqrt{\overline{w}}$, so the spectral distance
between two random graphs with the same average degree tends to zero as
their sizes tend to infinity. This result is also true for some scale-free
graphs. To elaborate this point, we now test our spectral distance in
simulations.

In order to compare two graphs in the same class with growing size, we use
the Barab{\'a}si-Albert scale-free model \cite{barabasi1999emergence,
  dorogovtsev2000structure} and the Erd{\H{o}}s-R{\'e}nyi random model
\cite{erdds1959random,erdHos1961evolution} as frames. We design the
following experiment. In the case of scale-free graphs, we generate two
groups of graphs from the same initial complete graph $H$ with small size
(say, 5). $G^{1}_{0}$ and $G^{2}_{0}$ are generated from $H$ independently
using preferential attachment, both of them have 1000 vertices. Then, we
produce two groups of scale-free graphs from $G^{1}_{0}$ and $G^{2}_{0}$
also with the preferential attachment, denoted by
$\{G^{1}_{0},G^{1}_{1},G^{1}_{2}, G^{1}_{3},\dots\}$ and
$\{G^{2}_{0},G^{2}_{1},G^{2}_{2}, G^{2}_{3},\dots\}$ respectively. In
Fig. \ref{F1:4_2_converage_2}(a), the $y$-coordinate shows the spectral
distance $D(G^{1}_{0},G^{j}_{i})$, where $j=1$ for the black curve with
$G^{1}_{0}$ and $G^{1}_{i}$ in the same group, and $j=2$ for the red curve
with $G^{1}_{0}$ and $G^{2}_{i}$ in different groups. The $x$-coordinate is
the size of $G^{j}_{i}$, denoted by $N(G^{j}_{i})$. At the beginning, the
distance $D(G^{1}_{0},G^{1}_{i})$ is smaller than $D(G^{1}_{0},G^{2}_{i})$,
but this difference disappears as the size of $G^{j}_{i}$ grows. We compare
this distance with the distance $D(G^{1}_{0},G^{3}_{i})$, where
$\{G^{3}_{i}\}$ are the random graphs with the same average degree and size
as $G^{1}_{i}$ and $G^{2}_{i}$. With the size growing, the distance
$D(G^{1}_{0},G^{3}_{i})$ is always larger than $D(G^{1}_{0},G^{1}_{i})$ and
$D(G^{1}_{0},G^{2}_{i})$. Similar results for random graphs can be seen in
Fig. \ref{F1:4_2_converage_2}(b), in which $\{G^{1}_{i}\}$ and
$\{G^{2}_{i}\}$ are two groups of random graphs, and $\{G^{3}_{i}\}$ is a
group of scale-free graphs. These results imply that the spectral distance
among graphs from different groups, but in the same class, approaches zero
as the size grows. In contrast, the spectral distance between graphs from
different classes, generated by different models and hence having different
structures, is always bigger than a certain value that is independent of
the graph size.

\begin{figure}[t]
\begin{center}
\includegraphics[width=4.8in]{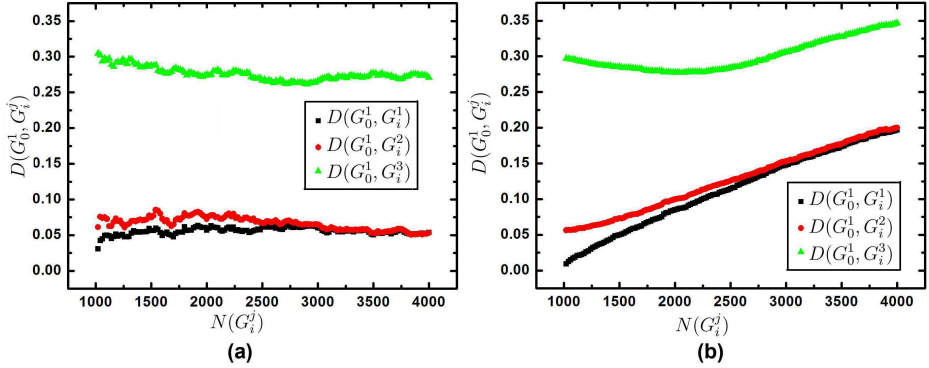}
\end{center}
\caption{The spectral distance among graphs in the same class with growing
  size, compared with that between graphs from different classes.}
\label{F1:4_2_converage_2}
\end{figure}

Similar results are obtained for $k-$regular trees. We consider $k-$regular
trees with different $k$ belonging to different small subclasses
independently of their size, for example, a $3-$ versus $4-$regular tree of
varying sizes. Starting from a small tree with 100 nodes and a fixed $k$,
we can get one group of $k-$regular trees through adding leaves. Here,
$T^{k}_{i}$ is a $k-$regular tree with $(i+1)\times 100$ vertices.

Fig. \ref{F1:4_2_converage_3}(a) plots the spectral distance among trees
with the same $k$ but different size. The $y-$coordinate is the spectral
distance $ D(T^{k}_{0}, T^{k}_{i})$, and the $x-$coordinate is the size of
$T^{k}_{i}$, denoted by $N(T^{k}_{i})$. For example, $D(T^{k}_{0},
T^{k}_{3})$ with $k=3$ means the spectral distance between two $3-$regular
trees of 100 vertices and $(3+1)\times 100=400$ vertices. All the curves
reach their highest values when $N(T^{k}_{i})=200$ ($i=1$ in the
simulation), but ultimately decrease close to zero. This shows the spectral
distance among two $k-$regular trees tends to zero when one of them is getting
larger while the other stays the same.

In contrast, the distance between two regular trees with different $k$, but
the same size, is bounded away from 0 independently of their
sizes. Fig. \ref{F1:4_2_converage_3}(b) shows the distance $
D(T^{k_{1}}_{i}, T^{k_{2}}_{i})$ for increasing size.

We conclude that the spectral distance between two large regular trees is
independent of their size, but depends on  the difference between
their degree $k$.
\begin{figure}[t]
\begin{center}
\includegraphics[width=4.8in]{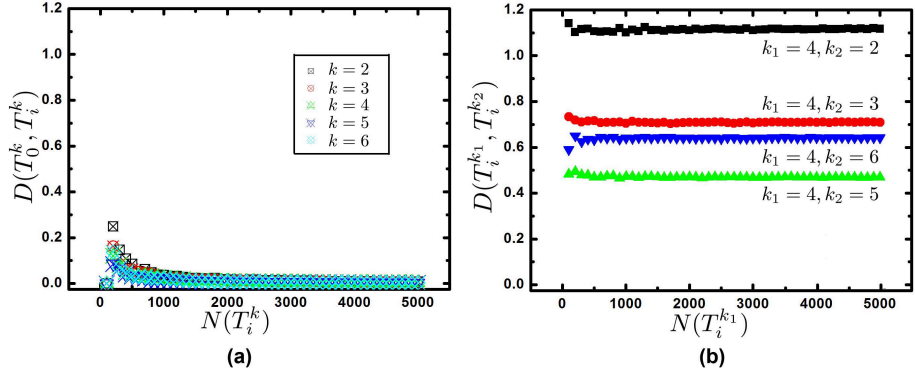}
\end{center}
\caption{The spectral distance between $k-$regular trees with growing
  size. (a) Trees from the same subclass, (b) Trees from different
  subclasses with $k_{1}=4$.}
\label{F1:4_2_converage_3}
\end{figure}

We finally compare these graph classes with each other and with some of the
graph classes such as complete (bipartite) graph, paths, cycles, or cubes,
that have been discussed in the previous section.

All the graphs are generated with almost equal size, around 2000 (except for the
$11-cube$ and $k-$regular graphs, the other graphs contain 2000
vertices). As mentioned before, the average degree can divide the
$k-$regular trees into some sub-classes, so we organzine groups of
graphs as their average degree ($\langle d\rangle$ is chosen from 4, 6,
8). The random graphs are obtained via the Erd{\H{o}}s-R\'{e}nyi random model
(see \cite{erdds1959random, erdHos1961evolution}), while the scale-free
graphs are obtained via the Barab\'{a}si-Albert scale-free model (see
\cite{barabasi1999emergence, dorogovtsev2000structure}). With fixed average
degree, 5 scale-free graphs and 5 random graphs are generated. For the  $n-$cube or $k-$regular trees, we choose the nearest number to 2000, for example, we
chose the $11-$cube, a $4-$regular tree with size 1457, a 6-regular tree
with size 4687 and a 8-regular tree with size 3201. We calculate the
spectral distance (Eq. \ref{e:distance}) among graphs in  one group, then
color values in the distance matrix, as shown in
Fig. \ref{F1:spectral_distance_1}.

In Fig. \ref{F1:spectral_distance_1}, a dark-blue square means that the spectral
distance is almost 0. With fixed average degree (Fig. \ref{F1:spectral_distance_1}(a)-(c)), such squares among the 5 random, the 5 scale-free graphs, between the cycles and paths and between the
 complete graph and the complete bipartite graph. Moreover, the
random graphs have small spectral distance to the scale-free graphs, and
also to the path and cycle. Fig. \ref{F1:spectral_distance_1}(d) shows all
the color squares above. In the group of random graphs, those with average
degree 4 are nearer to scale-free ones with average degree 4 than to random
ones with average degree 8, with respect to the spectral distance. A similar result obtains for the scale-free graphs. Thus, the
average degree can be a dominant factor when comparing graphs through the
spectral distance.

\begin{figure}[t]
\begin{center}
\includegraphics[width=5in]{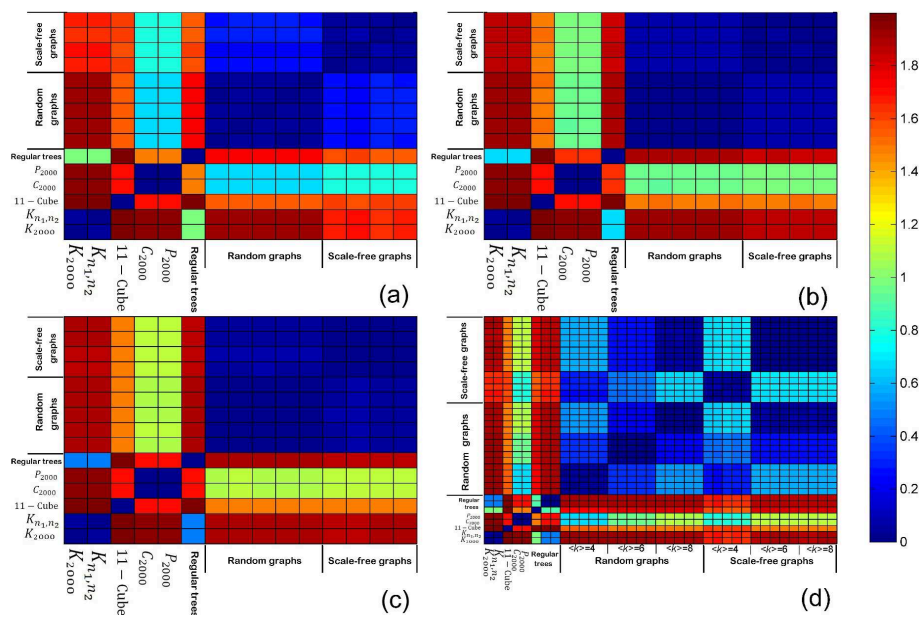}
\end{center}
\caption{The spectral distance between empirical graphs and particular
  graphs, whose sizes are around 2000 and degrees are $\langle d\rangle=$4,
  6, 8. (a) $\langle d\rangle=$4; (b) $\langle d\rangle=$6; (c) $\langle
  d\rangle=$8; (d) All plots taken from (a), (b) and (c).}
\label{F1:spectral_distance_1}
\end{figure}

In addition, we also take networks from empirical databases, such as
biological networks (domain-domain co-occurrence networks) and linguistic
networks (word-word co-occurrence networks). In biological networks, we
choose the species \textit{B.\ taurus} (2129 vertices, $\langle
d\rangle=11.22$), \textit{D.\ melanogaster} (1762 vertices, $\langle
d\rangle=9.51$), \textit{G.\ gallus} (1988 vertices, $\langle
d\rangle=10.85$), \textit{R.\ norvegicus} (2130 vertices, $\langle
d\rangle=11.01$) and \textit{S.\ scrofa} (1904 vertices, $\langle
d\rangle=9.98$). The random graphs and scale-free graphs with  similar
parameters (2000 vertices, $\langle d\rangle=10$) are selected for comparison. Fig. \ref{F1:spectral_distance_real}(a) is the distance matrix
composed of color squares. The biological networks are nearer to the
 random  and scale-free graphs than to the
regular graphs (such as cycle, path).

The linguistic networks (word-word co-occurrence networks) are generated
from the corpora (Wortschatz) database \cite{quasthoff2006corpus}. In contrast
 to biological networks, linguistic networks are quite dense, with average
degrees  around 500. We chose 5 linguistic networks with 2000 vertices:
German ($\langle d\rangle=512.83$), English ($\langle d\rangle=695.84$),
Danish ($\langle d\rangle=505.94$), Norwegian ($\langle d\rangle=490.04$)
and Swedish ($\langle d\rangle=481.7$). In
Fig. \ref{F1:spectral_distance_real}(b), linguistic networks are nearer to
the random graphs than to the other graphs. Since they are quite dense, as a typical
vertex is connected with around one quarter of the vertices in the whole
network, they are nearer to the complete graph and to the complete
bipartite graph than to the other regular graphs.

\begin{figure}[t]
\begin{center}
\includegraphics[width=5in]{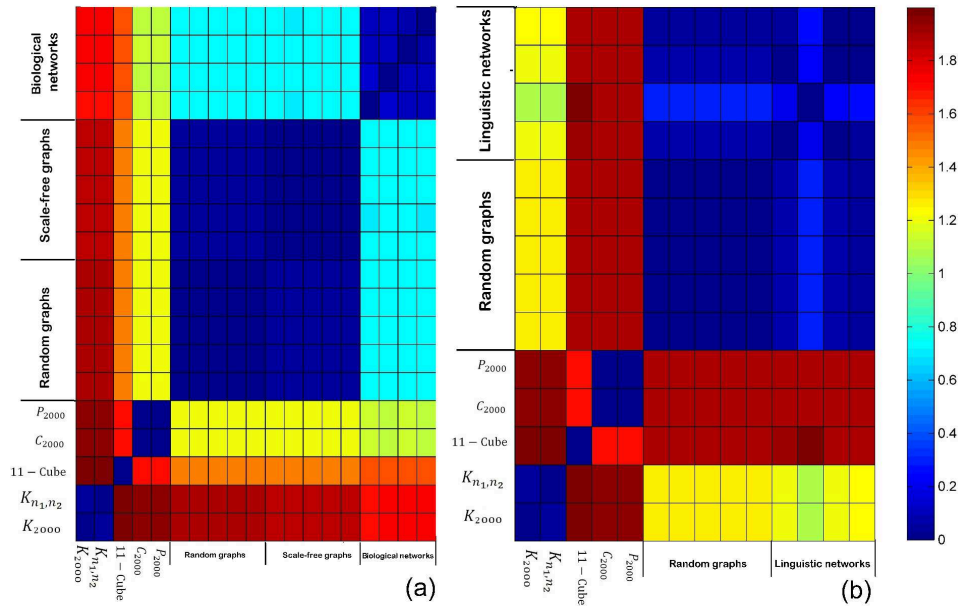}
\end{center}
\caption{The spectral distance between empirical graphs and networks from
  real database: (a) biological networks (domain-domain co-occurrence
  networks) (b) linguistic networks (word-word co-occurrence networks).}
\label{F1:spectral_distance_real}
\end{figure}

\section*{Acknowledgements}
JG was supported by the International Max Planck Research School of
Mathematics in the Sciences and thanks Dr.\ Qi Ding for useful
discussions. SL was partially supported by the EPSRC Grant EP/K016687/1.


\end{document}